\documentclass[sn-mathphys,Numbered]{sn-jnl}% Math and Physical Sciences Reference Style
\usepackage{graphicx}%
\usepackage{multirow}%
\usepackage{amsmath,amssymb,amsfonts}%
\usepackage{amsthm}%
\usepackage{mathrsfs}%
\usepackage[title]{appendix}%
\usepackage{xcolor}%
\usepackage{textcomp}%
\usepackage{manyfoot}%
\usepackage{booktabs}%
\usepackage{algorithm}%
\usepackage{algorithmicx}%
\usepackage{algpseudocode}%
\usepackage{listings}%

\usepackage{amsmath}
\usepackage{amssymb}
\usepackage{amsbsy}
\usepackage{graphicx}
\usepackage{amsfonts}
\usepackage{pgf,tikz, pifont}
\usepackage{amscd,amsmath,amsthm,amssymb}

\usepackage{pgf,tikz}
\usetikzlibrary{arrows}
\theoremstyle{thmstyleone}%
\newtheorem{theorem}{Theorem}%  meant for continuous numbers
\theoremstyle{thmstyletwo}%
\newtheorem{example}{Example}%
\newtheorem{remark}{Remark}%

\theoremstyle{thmstylethree}%

\begin{document}

\title[Article Title]{Correction to:  Powers of the vertex cover ideals (Collect. Math. 65 (2014)
169--181).}

%%=============================================================%%
%% Prefix	-> \pfx{Dr}
%% GivenName	-> \fnm{Joergen W.}
%% Particle	-> \spfx{van der} -> surname prefix
%% FamilyName	-> \sur{Ploeg}
%% Suffix	-> \sfx{IV}
%% NatureName	-> \tanm{Poet Laureate} -> Title after name
%% Degrees	-> \dgr{MSc, PhD}
%% \author*[1,2]{\pfx{Dr} \fnm{Joergen W.} \spfx{van der} \sur{Ploeg} \sfx{IV} \tanm{Poet Laureate}
%%                 \dgr{MSc, PhD}}\email{iauthor@gmail.com}
%%=============================================================%%

\author*[1]{\fnm{Danchneg} \sur{Lu}}\email{ludancheng@suda.edu.cn}

\author[1]{\fnm{Zexin} \sur{Wang}}\email{zexinwang6@outlook.com}

\equalcont{These authors contributed equally to this work.}

\affil[1]{\orgdiv{School of Mathematical Science}, \orgname{Soochow University}, \orgaddress{\street{No. 1 Shizhi Street}, \city{Suzhou}, \postcode{215021}, \state{Jiangsu}, \country{P.R.China}}}

%%==================================%%
%% sample for unstructured abstract %%
%%==================================%%

\abstract{ We point out an essential gap in the proof of one of main results in \cite{M} and then give a corrected proof for it.}


%%\pacs[JEL Classification]{D8, H51}

%%\pacs[MSC Classification]{35A01, 65L10, 65L12, 65L20, 65L70}

 \maketitle

  In \cite{M}  F. Mohammadi succeeded in classifying all the Cohen-
Macaulay cactus graphs.  Based on this result, he proved in \cite[Theorem 4.3]{M} that if $G$ is a Cohen-Macaulay  cactus graph then all the powers of the vertex cover ideal $J(G)$ of $G$ are weakly polymatroidal. In this note we  illustrate by an example that there exists an essential gap in the proof of
  \cite[Theorem 4.3]{M} and then present a corrected proof for this result.

We adopt the notation from the original paper \cite{M}, which  the reader should have  at  hand when  reading this  note. First of all, We   explain why some parts of the original proof of \cite[Theorem 4.3]{M} is not correct.

There were  several cases to consider in the proof. In case 4,   it is assumed that $z=y_{i3}$, then one deduces that $\{y_{i1}, y_{i2}, y_{i5}\}\subseteq \mathrm{supp}(g_j)$ for some $1\leq j\leq k$. Let us see the following example.
\vspace{2mm}

\begin{example} \em Let $G$  be a 5-cycle whose vertices are labeled $y_1,y_4,y_2,y_3,y_5$ in order. Then $f=y_1y_2y_3^2y_4y_5=(y_1y_2y_3)(y_3y_4y_5)$ and $g=y_1y_2y_3y_4^2y_5=( y_2y_4y_5)(y_1y_3y_4)$ are  minimal monomial generators of $J(G)^2$. We may write $g_1=y_2y_4y_5$ and $g_2=y_1y_3y_4$. Note that $\deg_{y_i}f=\deg_{y_i}g$ for $i=1,2$ and $\deg_{y_3}f>\deg_{y_3}g$. According to the proof of case 4 in \cite{M}, there exists some $g_j$ such that $\{y_{1}, y_{2}, y_{5}\}\subseteq \mathrm{supp}(g_j)$ and  so $y_3g/y_5\in J(G)^2$.   However, it is clear that $\{y_{1}, y_{2}, y_{5}\}\nsubseteq \mathrm{supp}(g_j)$ for $j=1,2$ and $y_3g/y_5\notin J(G)^2$. Furthermore, since $\frac{g}{y_1y_2y_5}$ does not belong to  $J(G)$ , it follows that $g$ can not be written as  $g=g'_1g'_2$  such that $g'_i\in J(G)$  for $i=1,2$ and $\{y_{1}, y_{2}, y_{5}\}\subseteq \mathrm{supp}(g'_i)$ for some $i$.
\end{example}
\vspace{2mm}

This example shows the proof of case 4 is not correct. The proof of case 5 is not correct in a similar way. In the sequel, we give  revised proofs for cases 4 and 5.

\vspace{2mm}
\begin{theorem}\label{thm1} {\em(\cite[Theorem 4.3]{M})}\quad   Let $G$ be a Cohen-Macaulay cactus graph. Then $J(G)^k$ is weakly polymatroidal for all $k\geq 1$.
\end{theorem}

\begin{proof}  Recall that  a cactus graph is a  connected graph  in which each edge  belongs to at most
one cycle.  As in the original proof of  \cite[Theorem 4.3]{M}, we may assume that $V(G)$ is the disjoint union of $$F_1,\ldots,F_m,G_1,\ldots,G_n,L_1,\ldots,L_t$$
such that the induced graph of $G$ on $F_i$ is a clique of size 3 or 2, the induced graph of $G$ on $G_j$ is a  basic 5-cycle, and $L_l$ is an edge which belongs to a 4-cycle of $G$  for each $i,j,l$. Here, a 5-cycle of a graph $G$ is {\it basic} if   it does not contain two adjacent vertices of degree 3 or more  in $G$, see \cite{HMT}.

Consider the following ordering  of the variables (corresponding to the vertices of $G$): $$x_{11} >\cdots> x_{1k_1}>\cdots> x_{m1}>\cdots>x_{mk_m}>y_{11}>\cdots>y_{15}>$$ $$\cdots>y_{n1}>\cdots>y_{n5}>z_{11}>z_{12}>\cdots>z_{t1}>z_{t2},$$ where $F_{i}=\{{x_{i1},\ldots,x_{ik_i}}\}, 2\leq k_i\leq 3$ and ${x_{ib_i},\ldots ,x_{ik_i}}$ are the free vertices of $F_{i}$ for all $i$, and $y_{j1},y_{j4},y_{j2},y_{j3},y_{j5}$ are the vertices of $G_{j}$ successively such that the vertices $y_{j3},y_{j4},y_{j5}$ are of degree two for all $j$, and $z_{i1},z_{i2}$  are the vertices of $L_{i}$ of degree two for all $i$.
 Note that  for each $i\in [n]$, each minimal vertex cover $C$ of $G$ contains exactly one of the following subsets: $$\{y_{i1},y_{i2},y_{i5}\},\{y_{i2},y_{i4},y_{i5}\},\{y_{i1},y_{i2},y_{i3}\},\{y_{i1},y_{i3},y_{i4}\},\{y_{i2},y_{i4},y_{i5}\}$$

Now, let $f=f_{1} \cdots f_{k}$ and $g=g_{1} \cdots g_{k}$ be  monomials in the minimal generating set of $J(G)^{k}$ with $f_s, g_s\in J(G)$ for $1\leq s\leq k$.  Suppose that there exists a variable $z$ such that $\mathrm{deg}_{z^{'}}f=\mathrm{deg}_{z^{'}}g$ for every variable $z^{'}>z$ and $\mathrm{deg}_{z}f>\mathrm{deg}_{z}g$. We only need to find a variable $w<z$ such that $zg/w\in J(G)^k$ in the  cases where $z=y_{i3}$ and where $z=y_{i4} \mbox{ or } y_{i5}$ for some $i$, which correspond to  case 4 and case 5 of the original proof, respectively. To this end, we introduce some additional notation.

  Let $i$ be the number mentioned above.   First,  for each  triple $\{a,b,c\}$ such that $\{y_{ia},y_{ib}, y_{ic}\}$ is a minimal vertex cover of $G_i$, we set $$f_{abc}=|\{1\leq s\leq k:\{y_{ia},y_{ib},y_{ic}\}\subseteq \mathrm{supp}(f_s) \}|$$  and
   $$ g_{abc}=|\{1\leq s\leq k:\{y_{ia},y_{ib},y_{ic}\}\subseteq \mathrm{supp}(g_s) \}|.$$

Next, we denote by $T$  the induced subgraph of $G$  on the  subset $V(G)\setminus \{y_{i3},y_{i4},y_{i5}\}$, and for $a=1,2,$ we set
 $$T_a=\{z\in V(T)\setminus \{y_{i1}, y_{i2}\}\:\; \mbox{ there exists a path in } T  \mbox{ connecting } z  \mbox{ with } y_{ia}\}.$$
  Since $G$ is a cactus graph,  we have   the subsets $T_1, \{y_{i1},y_{i2}\}$ and $T_2$ are pairwise disjoint and  $V(T)=T_1\cup \{y_{i1},y_{i2}\} \cup T_2$.

\vspace{2mm}

\noindent {\it Case } $z=y_{i3}$:   If $g_{125}\neq0$, there exists $s \in \{1,\cdots, k\}$ such that $\{y_{i1},y_{i2},y_{i5}\}\subseteq \mathrm{supp}(g_s)$. Since the subset $(\mathrm{supp}(g_s)\backslash\{y_{i5}\})\cup\{y_{i3}\}$ is again a minimal vertex cover of $G$, we have $y_{i3}g/y_{i5} \in J(G)^k$ and so we are done. Note that  $g_{125}\neq 0$ if $k=1$, we  assume now that $g_{125}=0$ and $k\geq2$.

Since  $\mathrm{deg}_{y_{ia}}f=\mathrm{deg}_{y_{ia}}g$ for $a=1,2$  and $\mathrm{deg}_{y_{i3}}f>\mathrm{deg}_{y_{i3}}g$,  we have the following formulas:

 \begin{itemize}\item $f_{134}+f_{123}+f_{125}=g_{134}+g_{123}+g_{125}$;    \hfill\ding{202}
 \item $f_{245}+f_{123}+f_{125}=g_{245}+g_{123}+g_{125};$ \hfill\ding{203}
 \item $f_{134}+f_{123}+f_{345}>g_{134}+g_{123}+g_{345};$  \hfill\ding{204}
 \item $f_{134}+f_{245}+f_{123}+f_{125}+f_{345}=g_{134}+g_{245}+g_{123}+g_{125}+g_{345}=k.$ \hfill\ding{205}

  \end{itemize}

  Subtracting \ding{202} from  \ding{205} and  subtracting \ding{203} from  \ding{205}, we obtain
  \begin{itemize}
   \item   $f_{245}+f_{345}=g_{245}+g_{345}$; \hfill \ding{206}

       \item    $f_{134}+f_{345}=g_{134}+g_{345}$. \hfill \ding{207}

        \end{itemize}

   On the other hand, by combining  \ding{204} and \ding{202} with the assumption $g_{125}=0$, we have
    $$f_{134}+f_{123}+f_{345}>g_{134}+g_{123}+g_{345}=f_{134}+f_{123}+f_{125}+g_{345},$$
    and so  $f_{345}>f_{125}+g_{345}\geq g_{345}$. From this together with  \ding{206} and \ding{207},  it follows that
     $f_{134}<g_{134}$ and $f_{245}<g_{245}$. This particularly implies  $g_{134}>0$  and $g_{245}>0$, and thus, there exist $p,q\in \{1,\cdots, k\}$ with $p\neq q$ such that $\{y_{i2},y_{i4},y_{i5}\}\subseteq g_p$ and  $\{y_{i1},y_{i3},y_{i4}\}\subseteq g_q$.  Say $p=1$ and $q=2$. To prove $y_{i3}g/y_{i4}$ belongs to  $J(G)^k$, we only need to prove that $y_{i3}g_1g_2/y_{i4}$ belongs to $J(G)^2$.

 We  may write $g_1, g_2$ as follows: $$g_1=y_{i2}y_{i4}y_{i5}u_1u_2 \mbox{\quad and\quad }  g_2=y_{i1}y_{i3}y_{i4}v_1v_2,$$  where  $\mathrm{supp}(u_i)\cup \mathrm{supp}(v_i)\subseteq T_i$  for $i=1,2$. Then there are the following decomposition: $$y_{i3}g_1g_2/y_{i4}=(y_{i3}y_{i4}y_{i5}u_1u_2)(y_{i1}y_{i2}y_{i3}v_1v_2)=(y_{i3}y_{i4}y_{i5}u_1v_2)(y_{i1}y_{i2}y_{i3}v_1u_2).$$  We next check  that both $\mathrm{supp}(y_{i3}y_{i4}y_{i5}u_1v_2)$ and $\mathrm{supp}(y_{i1}y_{i2}y_{i3}v_1u_2)$ are vertex covers of $G$. To this end, we denote $V_1=\mathrm{supp}(y_{i3}y_{i4}y_{i5}u_1v_2)$ and $V_2=\mathrm{supp}(y_{i1}y_{i2}y_{i3}v_1u_2)$.  Let $e=\{z_1,z_2\}$ be an edge of $G$.

  If $e\subseteq V(G_i)$, then it is clear that $V_1\cap e\neq \emptyset$ and $V_2\cap e\neq \emptyset$.

  If $z_1=y_{i1}$ and $z_2\in T_1$ then $z_2\in \mathrm{supp}(u_1)\cap e$. This  implies $V_1\cap e\neq \emptyset$ and $V_2\cap e\neq \emptyset$; Similarly, we have if $z_1=y_{i2}$ and  $z_2\in T_2$ then $V_1\cap e\neq \emptyset$ and $V_2\cap e\neq \emptyset$.

 Finally  suppose that $e\cap V(G_i)=\emptyset$. Then either $e\subseteq  E(T_1)$ or $e\subseteq E(T_2)$. In the case that $e\subseteq  E(T_1)$, we have $\mathrm{supp}{(u_1)}\cap e\neq \emptyset$ and $\mathrm{supp}{(v_1)}\cap e\neq \emptyset$  and so  $V_1\cap e\neq \emptyset$ and $V_2\cap e\neq \emptyset$. The case that $e\subseteq  E(T_2)$ can be proved similarly.

 Thus, both $V_1$ and $V_2$ are vertex covers of $G$ indeed and so $y_{i3}g_1g_2/y_{i4}$ belongs to $J(G)^2$, as required.
 \vspace{2mm}

 \noindent {\it  Case} $z=y_{i4}\mbox{ or } y_{i5}$:  Since the case where $z=y_{i5}$ is impossible actually, we only need to consider the case where  $z=y_{i4}$. Moreover we may assume that $k\geq2$ as in the  case $z=y_{i3}$.  Because  $\mathrm{deg}_{y_{ia}}f=\mathrm{deg}_{y_{ia}}g$ for $a=1,2,3$  and $\mathrm{deg}_{y_{i4}}f>\mathrm{deg}_{y_{i4}}g$,  we have the following formulas:
\begin{itemize}\item $f_{134}+f_{123}+f_{125}=g_{134}+g_{123}+g_{125}$;    \hfill\ding{202}
 \item $f_{245}+f_{123}+f_{125}=g_{245}+g_{123}+g_{125};$ \hfill\ding{203}
 \item $f_{134}+f_{123}+f_{345}=g_{134}+g_{123}+g_{345};$  \hfill\ding{204}
 \item $f_{134}+f_{245}+f_{345}>g_{134}+g_{245}+g_{345};$  \hfill\ding{205}
 \item $f_{134}+f_{245}+f_{123}+f_{125}+f_{345}=g_{134}+g_{245}+g_{123}+g_{125}+g_{345}=k.$ \hfill\ding{206}

  \end{itemize}

  Subtracting \ding{202}, \ding{203} and  \ding{204} from  \ding{206} respectively, we have
  \begin{itemize}
   \item   $f_{245}+f_{345}=g_{245}+g_{345}$; \hfill \ding{207}
    \item    $f_{134}+f_{345}=g_{134}+g_{345}$; \hfill \ding{208}
     \item    $f_{245}+f_{125}=g_{245}+g_{125}$. \hfill \ding{209}
        \end{itemize}

  We also obtain $f_{245}>g_{245}$ by substituting  \ding{208} into \ding{205}.
  From this together with \ding{207} and \ding{209}, it follows that $f_{345}<g_{345}$ and $f_{125}<g_{125}.$ In particular, $g_{345}>0$ and $g_{125}>0$. Consequently, there exist $p,q\in \{1,\cdots, k\}$ such that $\{y_{i3},y_{i4},y_{i5}\}\subseteq g_p$ and $\{y_{i1},y_{i2},y_{i5}\}\subseteq g_q$.  Say $p=1$ and $q=2$. It suffices to prove that the monomial $y_{i4}g_1g_2/y_{i5}$ belongs to $J(G)^2$.

For this, we  decompose $g_1$ and $g_2$ as follows: $$g_1=y_{i3}y_{i4}y_{i5}u_1u_2 \mbox{\quad and\quad }g_2=y_{i1}y_{i2}y_{i5}v_1v_2 ,$$ where $\mathrm{supp}(u_i)\cup\mathrm{supp}(v_i)\subseteq T_i$  for $i=1,2$.  This implies $$y_{i4}g_1g_2/y_{i5}=(y_{i3}y_{i4}y_{i5}u_1u_2)(y_{i1}y_{i2}y_{i4}v_1v_2)=(y_{i2}y_{i4}y_{i5}u_1v_2)(y_{i1}y_{i3}y_{i4}v_1u_2).$$
   We can check  that both $\mathrm{supp}(y_{i2}y_{i4}y_{i5}u_1v_2)$ and $\mathrm{supp}(y_{i1}y_{i3}y_{i4}v_1u_2)$ are vertex covers of $G$, as in the case where $z=y_{i3}$. Hence $y_{i4}g_1g_2/y_{i5}$ belongs to $J(G)^2$. This finishes the proof. \end{proof}

\vspace{2mm}

\begin{remark} \em The method used in our proof of this note depends heavily on the condition that every edge of $G$  belongs to at most cycle of $G$. The gap of the original proof in \cite{M} was found when we tried to use the idea in the proof to show that $J(G)^k$ is weakly polymatroidal for all $k\geq 1$ if $G$ is a Cohen-Macaulay graph of large girth, whose structure was described explicitly in \cite{HMT}.  But we failed in this regard and we can only  prove in \cite[Proposition 3.6]{LW} that $J(G)$ is weakly polymatroidal for such graphs.
\end{remark}

\section*{Declarations}

\begin{itemize}
\item Funding: This work was supported by National Natural Science Foundation of China  (Grant No. 11971338)
\item The authors have no competing interests to declare that are relevant to the content of this article.

\item We would like to express our sincere thanks to  the referee, whose comments improved the presentation of the paper greatly.
\end{itemize}

\end{document}